\documentclass[11pt]{article}

\usepackage{amsfonts, amssymb}
\usepackage{amsmath}
\usepackage{verbatim}
\usepackage{amsthm}
\usepackage{mathrsfs}   
\usepackage{enumerate}

\usepackage{indentfirst}
\usepackage{color}
\renewcommand{\l}{\hat{l}}
\renewcommand{\L}{\hat{L}}
\newcommand{\E}{\hat{E}}
\newcommand{\e}{\hat{e}}
\newcommand{\p}{\prime}

\newcommand{\n}{\mathbf{n}}
\newcommand{\la}{\lambda}

\begin{document}

\newtheorem{lemma}{Lemma}[section]
\newtheorem{proposition}[lemma]{Proposition}
\newtheorem{theorem}[lemma]{Theorem}
\newtheorem{corollary}[lemma]{Corollary}
\newtheorem{definition}[lemma]{Definition}
\newtheorem{example}[lemma]{Example}
\newtheorem{remark}[lemma]{Remark}
\numberwithin{equation}{section}


\title{Modules and Structures of Planar Upper\\ Triangular Rook Monoids}
\date{}
\author {Jianqiang Feng  ~~Wenli Liu  ~~Ximei Bai ~~Zhenheng Li}

\vspace{ -8mm}
\maketitle

\begin{abstract}In this paper, we discuss modules and structures of the planar upper triangular  rook monoid $B_n$. We first show that the  order of $B_n$ is a Catalan number, then we investigate the properties of a module $V$ over $B_n$ generated by a set of elements $v_S$ indexed by the power set of $n$. We find that every nonzero submodule of $V$ is cyclic and completely decomposable; we give a necessary and sufficient condition for a submodule of $V$ to be indecomposable. We show that every irreducible submodule of $V$ is $1$-dimensional. Furthermore, we give a formula for calculating the dimension of every submodule of $V$. In particular, we provide a recursive formula for calculating the dimension of the cyclic module generated by $v_S$, and show that some dimensions are Catalan numbers, giving rise to new combinatorial identities.

\vspace{ 0.2cm} \noindent {\bf Keywords:} Rook monoid, order preserving, order decreasing, module, Catalan number, generators and relations.

\vspace{ 0.2cm} \noindent {\bf 2010 AMS Subject Classification:} 20M32, 05E10

\end{abstract}

\baselineskip 14pt
\parskip 1mm
\section{Introduction}

A matrix is a rook matrix if each entry is $0$ or $1$ and each row and column have at most one $1$. A rook matrix $A$ is {\it planar} or {\it order preserving} if the matrix obtained from $A$ by deleting all the zero rows and all the zero columns is an identity matrix.
The structure and representation theory of the rook monoid, consisting of all rook matrices, are intensively studied \cite{M2, S}. Herbig gives a structure and representation theory of a planar rook monoid \cite{H}. The planar upper triangular  rook monoid $B_n$ consists of planar upper triangular  rook matrices of size $n$.

It is natural to ask: What are the representation and structure properties of the planar upper triangular  rook monoid? More specifically, how do we construct interesting modules over $B_n$, and what do irreducible $B_n$-modules look like? What is the order of $B_n$ and what are the dimensions of the modules of interest? How are the order and the dimensions related to combinatorics? What are the generators and defining relations of $B_n$? These questions are closely related to the theory of linear algebraic monoids, since it was made clear in \cite{LLC14, R1} that we are here dealing with the most familiar interesting case of planar upper triangular Renner monoids of reductive monoids. For more information on Renner monoids, see \cite{LR03, LLC14, P1, R2, S}.

In this paper we answer the questions above, and our discussion goes a little deeper, showing that the $B_n$-module properties of $V$ are dramatically different from those of $V$ as a module over the planar rook monoid.
In Section 2 after gathering basic definitions and concepts related to planar upper triangular  rook monoids $B_n$, we give a new interpretation of $B_n$ using generalized reduced echelon matrices. We then calculate in Section 3 the order of $B_n$ in two different ways and show that it is a Catalan number.

Section 4 is devoted to the investigation of $B_n$-modules over a field $F$ of characteristic $0$. Let $V_k$ be a vector space over $F$ generated by a set of elements $v_S$ indexed by the $k$-subsets of $\n=\{1, ..., n\}$. Then $V_k$ is a $B_n$-module under the action (\ref{moddef}). We are particularly interested in $B_n$-submodules of $V_k$ and of $V=\bigoplus_{k=0}^n V_k$. We find that
every nonzero submodule of $V$ is completely decomposable, and that a submodule of $V$ is indecomposable if any only if it lies in some $V_k$. Furthermore, we show that every submodule of $V$ is cyclic, and that each irreducible submodule of $V$ is $1$-dimensional and is contained in all nonzero submodules of some $V_k$. We also show that any two different submodules of $V$ are not isomorphic. Moreover, we give a formula for calculating the dimension of every submodule of $V$ using the inclusion-exclusion principle. In particular, we provide a recursive formula for calculating the dimensions of the modules generated by a single basis vector, and find that some of these dimensions are Catalan numbers again, connecting to combinatorics. Viewed as $B_t$-modules with $t<n$, we are able to decompose some $B_n$-submodules of $V_k$ into indecomposable $B_t$-submodules.
Section 5 describes the generators and defining relations of $B_n$.

{\bf Acknowledgement} {We would like to thank Dr. M. Can for useful email communications and Dr. R. Koo for valuable comments.}

\section{Preliminaries}

\begin{definition}
An {\em injective partial map} $f$ of
$\mathbf{n}$ is a one-to-one map of a subset $D(f)$
of $\mathbf{n}$ onto a subset $R(f)$ of $\mathbf{n}$ where $D(f)$ is
the domain of $f$ and $R(f)$ is the range of $f$.
\end{definition}
We agree that there is a map with empty domain and range and call it 0 map. We can write an injective partial map $f$ of $\mathbf{n}$ in 2-line notation by writing the numbers $s_1,\dots, s_k$ in the top line if $D(f)=\{s_1,\dots, s_k\}$, and then below each number we write its image. Equivalently, we can represent such a map by an $n\times n$ rook matrix, where the entry in the
$i$th row and the $j$th column is 1 if the map takes $j$ to $i$, and is 0 otherwise.
For example, the map $\sigma$ given below is an injective partial map of $\mathbf{5}$,
{
\begin{eqnarray*}
\sigma&=& \left(
            \begin{array}{cccc}
              1 & 2 & 3 & 5 \\
              1 & 2 & 4 & 5 \\
            \end{array}
          \right)
 \\
   &=& \left(
         \begin{array}{ccccc}
           1 & 0 & 0 & 0 & 0 \\
           0 & 1 & 0 & 0 & 0 \\
           0 & 0 & 0 & 0 & 0 \\
           0 & 0 & 1 & 0 & 0 \\
           0 & 0 & 0 & 0 & 1 \\
         \end{array}
       \right)~.
\end{eqnarray*}
}
\begin{definition} The {\em rook monoid} $R_n$ is the monoid of
injective partial maps from $\mathbf{n}$ to $\mathbf{n}$, whose operation is the composition of partial maps and the identity element is the identity map of $\n$.
\end{definition}
Since elements of $R_n$ are not necessarily invertible, $R_n$ is not a group. The map with empty domain and empty range behaves as
a zero element. In matrix form, the composition of $R_n$ is consistent with the usual matrix multiplication. Here is an example: for $g=\left(
          \begin{array}{ccc}
            2 & 3 & 4 \\
            1 & 5 & 2 \\
          \end{array}
        \right), ~
f=\left(
       \begin{array}{cccc}
         1 & 3 & 4 & 5 \\
         1 & 2 & 3 & 4 \\
       \end{array}
     \right)\in R_5,
$ we have
$$
gf=
\left(
          \begin{array}{ccc}
            2 & 3 & 4 \\
            1 & 5 & 2 \\
          \end{array}
        \right)\circ
\left(
       \begin{array}{cccc}
         1 & 3 & 4 & 5 \\
         1 & 2 & 3 & 4 \\
       \end{array}
     \right)
=\left(
   \begin{array}{ccc}
     3 &4& 5 \\
     1 &5& 2 \\
   \end{array}
 \right)~.
$$
The corresponding matrix form of the operation reads as
$$
gf= \left(
  \begin{array}{ccccc}
    0 & 1 & 0 & 0 & 0 \\
    0 & 0 & 0 & 1 & 0 \\
    0 & 0 & 0 & 0 & 0 \\
    0 & 0 & 0 & 0 & 0 \\
    0 & 0 & 1 & 0 & 0 \\
  \end{array}
\right) \left(
  \begin{array}{ccccc}
    1 & 0 & 0 & 0 & 0 \\
    0 & 0 & 1 & 0 & 0 \\
    0 & 0 & 0 & 1 & 0 \\
    0 & 0 & 0 & 0 & 1 \\
    0 & 0 & 0 & 0 & 0 \\
  \end{array}
\right)= \left(
  \begin{array}{ccccc}
    0 & 0 & 1 & 0 & 0 \\
    0 & 0 & 0 & 0 & 1 \\
    0 & 0 & 0 & 0 & 0 \\
    0 & 0 & 0 & 0 & 0 \\
    0 & 0 & 0 & 1 & 0 \\
  \end{array}
\right)~. $$

An injective partial map from $\mathbf{n}$ to $\mathbf{n}$ is {\em order preserving} if whenever $a<b$ in the domain of the map, then $f(a)<f(b)$. An injective partial map $f$ is order preserving if and only if the matrix obtained from the matrix form of $f$ by deleting all the zero rows and all the zero columns is an identity matrix; equivalently the graph obtained from the 2-line notation of $f$ by joining all defined $f(a)$ in the range of the map to $a$ is a planar graph, which justifies the name in the following definition.
\begin{definition}
The {\em planar rook monoid}, denoted by $PR_n$, is the monoid of {\em order preserving} injective partial maps from $\mathbf{n}$
to $\mathbf{n}$.
\end{definition}
\noindent Obviously, $PR_n$ is a submonoid of $R_n$. The structure and representation of the planar rook monoid is studied in Herbig \cite{H}. In particular, $V_k$ is an irreducible $PR_n$-module.

The next definition will give a different interpretation of an order preserving injective partial map.
\begin{definition}
A rectangular matrix is a {\em generalized (row and column) reduced echelon matrix if}
\vspace{-3mm}
\begin{enumerate}[{\rm(1)}]
  \item Each leading entry of a row is $1$ and is in a column to the right of the leading entry of the row above it.
\vspace{-3mm}
  \item Each leading entry of a column is $1$ and is in a row below the leading entry of the column to the left of it.
\vspace{-3mm}
  \item Each leading 1 is the only nonzero entry in its column and its row.
\end{enumerate}
\end{definition}
\noindent
This definition does not require that all nonzero rows are above any zero rows nor all nonzero columns are to the left of any zero columns.  Since the row and column reduced echelon form of a matrix is equivalent to the normal form of the matrix, we can consider a generalized reduced echelon matrix to be a generalization of the normal form of a matrix.

An injective partial map is order preserving if and only if its matrix form is a generalized reduced echelon matrix. Thus, the set of all the generalized reduced echelon matrices of size $n$ is a monoid with respect to the multiplication of matrices, and the order of this monoid is $\binom{2n}{n}$, since the order of $PR_n$ is
\[
    |PR_n| = \binom{2n}{n}~.
\]

An injective partial map is called {\em order decreasing} if for all $a$ in the domain of the map, we have $f(a)\leq a$. Equivalently, an injective partial map is order decreasing if and only if its matrix form is an upper triangular rook matrix, which motivates the name in the following definition.

\begin{definition}
The {\em planar upper triangular  rook monoid}, denoted by $B_n$, is the monoid of order preserving, order decreasing injective partial maps from $\mathbf{n}$ to $\mathbf{n}$.
\end{definition}

An injective partial map is in $B_n$ if and only if its matrix form is an upper triangular generalized reduced echelon matrix. In the previous example, we have $f\in B_5$, but  $g\notin B_5$.

\section{Order of $B_n$}

We first show that the order of $B_n$ is a Catalan number, which is defined by
$c_0=c_1=1$ and $c_n=\sum_{i=0}^{n-1}c_ic_{n-1-i}$ for
$n>1$ (see \cite{St}).
\begin{proposition}\label{bnCatalan}
 Let $n\ge 0$. Then the order of the planar upper triangular rook monoid $B_n$
 is the Catalan number $c_{n+1}$, that is, $b_n=c_{n+1}$.
\end{proposition}
\begin{proof}
To prove the proposition, we set up a one-to-one correspondence
between the set $B_n$ and the set $C_{n+1}$ of all sequences
$a_1,a_2,\dots,a_{2n+2}$ of $n+1$ copies of 1's and $n+1$ copies of $-1$'s, such that
$a_1+a_2+\dots+a_l\geq 0$ for all $1\leq l\leq 2n+2$.

Let $f$ be an element of $B_n$
with domain $S=\{s_1<s_2<\dots<s_k\}$ and range
$T=\{t_1<t_2<\dots<t_k\}$. Define $s_1^\p=s_1, \,s_i^\p=s_i-s_{i-1}$
for $2\leq i\leq k$, and $s_{k+1}^\p=n+1-s_k$. Also define
$t_1^\p=t_1, \,t_i^\p=t_i-t_{i-1}$ for $2\leq i\leq k$, and
$t_{k+1}^\p=n+1-t_k$. Let $f^\p$ be the sequence
\begin{equation}\label{sequence}
\underbrace{1,\,\dots,\,1}_{s_1^\p},\, \underbrace{-1,\,\dots,\,-1}_{t_1^\p},\,
\underbrace{1,\,\dots,\,1}_{s_2^\p}, \underbrace{-1,\dots,-1}_{t_2^\p}, \,
\dots, \,
\underbrace{1,\,\dots,\,1}_{s_{k+1}^\p},\, \underbrace{-1,\,\dots,\,-1}_{t_{k+1}^\p} ~.
\end{equation}
Now we prove $f^\p\in C_{n+1}$.  By definition of $B_n$, we have
$t_i\leq s_i$ for all $1\leq i\leq n$. Thus
$s_1^\p-t_1^\p=s_1-t_1\geq 0, \,(s_1^\p+s_2^\p+\dots+s_i^\p)
-(t_1^\p+t_2^\p+\dots+t_i^\p)=s_i-t_i\geq 0$ for $2\leq i\leq k$ and
$(s_1^\p+s_2^\p+\dots+s_{k+1}^\p)
-(t_1^\p+t_2^\p+\dots+t_{k+1}^\p)=(n+1)-(n+1)\geq 0$. Denote the
partial sum of the first $l$ items in (\ref{sequence}) by $a_l$.
Then our previous argument shows that $a_l\geq 0$ for
$l=s_1+t_1+s_2+t_2+\dots+s_h+t_h$, where $1\leq h\leq k+1$, and this
implies $a_l\geq 0$ for $1\leq l\leq 2+2n$ by the format of
(\ref{sequence}). Hence, $f^\p\in C_{n+1}$.

We next show that the mapping from $B_n$ to $C_{n+1}$ defined by
\begin{eqnarray*}
   \sigma:B_n&\rightarrow& C_{n+1} \\
  f &\mapsto& f^\p
\end{eqnarray*}
is bijective. From the definition of the map $\sigma$, it is straightforward to
see that $\sigma$ is injective. Now we prove the onto property of
the map $\sigma$. For arbitrary element
$f^\p=\{a_1,a_2,\dots,a_{2n+2}\}$ of $C_{n+1}$, the condition
$a_1+a_2+\dots+a_l\geq 0$ for all $1\leq l\leq 2n+2$ implies $a_1=1$
and $a_{2n+2}=-1$. As in (\ref{sequence}), let $s_1^\p$ be the number
of the first consecutive 1's in $f^\p$, let $t_1^\p$ be the number of
the consecutive $-1$'s that follow, and define similarly $s_i^\p$ and
$t_i^\p$ as before for $2\leq i\leq k+1$. Let
$s_1=s_1^\p,\, t_1=t_1^\p,\,  \,s_i=s_1^\p+s_2^\p+\dots+s_i^\p$, and $t_i=t_1^\p+t_2^\p+\dots+t_i^\p$ for $2\leq i\leq k$. Then the condition $a_1+a_2+\dots+a_l\geq 0$ for all
$1\leq l\leq 2n+2$ implies $t_i\leq s_i$ for $1\leq i\leq k$.
Define $f$ to be the mapping $f(s_i) = t_i$ with domain $s_1,s_2,\dots,s_k$
and range $t_1,t_2,\dots,t_k$. Then
$f\in B_n$ and $\sigma(f)=f^\p$. Hence $\sigma$
is a one-to-one correspondence between the set $B_n$ and the set
$C_{n+1}$. It is well known that the order of $C_{n+1}$ is the
Catalan number $c_{n+1}$, so the proposition is proved.
\end{proof}
\begin{remark}
{\rm
    Clearly $B_n$ is a submonoid of the monoid $\mathscr B_n$ consisting of all order decreasing (not necessarily planar) injective partial maps. See \cite{BRR} for the order of $\mathscr B_n$.
}
\end{remark}

The following corollary is immediate.
\begin{corollary}
    The number of upper triangular generalized reduced echelon matrices of size $n$ is the Catalan number $c_{n+1}$.
\end{corollary}

Our next proposition provides a recursive formula for calculating the order $b_n$ of
\begin{equation*}\label{bn}
B_n=\{f\in R_n \mid  f(j)\leq j \;\mathrm{for}\; j\in D(f);\, f(i)<f(j)
\;\mathrm{for}\; i,j\in D(f) \;\mathrm{and}\; i<j\}~.
\end{equation*}
For $0\leq p,\,q\leq n-1$, let $b_{p,\,q}=|B_{p,\,q}|$ where
\[
    B_{p,\,q}=\{f\in B_n \mid  D(f)\subseteq \{n-q,\, \dots,\,n\} \text{ and }R(f)\subseteq \{n-p,\, \dots,\, n\}\}~.
\]
\begin{proposition}\label{bnre} Let $n\ge 1$. Then
\vspace{-2mm}
\begin{enumerate}[{\rm(1)}]
  \item $b_0=1$,\, $b_1=2$~.
  \vspace{-2mm}
  \item $b_{p,\,0}=p+2$\hspace{3.5cm} for $0\le p\le n-1$~.
  \vspace{-2mm}
  \item $b_{p,\,p}=b_{p+1}$\quad\quad\quad\quad\quad\quad\quad\quad\quad\, for $0\le p\le n-1$~.
  \vspace{-2mm}
  \item $b_n = 2b_{n-1} + 1 + \sum_{q=0}^{n-3}b_{n-2,\,q}$\quad\, for $n\geq 2$~.
  \vspace{-2mm}
  \item $b_{p,\,q}=1+\sum_{r=0}^q b_{p-1,\,r}$\quad\quad\quad\quad\, for $1\leq q<p\leq n-1$~.
\end{enumerate}
\end{proposition}
\begin{proof}
Parts (1), (2), and (3) are clear. To prove (4), divide the elements of $B_n$ into two groups: the elements whose ranges contain $1$, and those whose ranges do not contain $1$. Part (4) follows from the following three identities:
\vspace{-2mm}
\begin{align*}
  b_{n-1}&=|\{f\in B_n \mid  f(1)=1\}| = |\{f\in B_n \mid  1\notin R(f)\}|~. \\
  b_{n-2,\,q}&=|\{f\in B_n\mid f(n-1-q)=1\}|\quad \text{for}\quad 0\leq q\leq n-3 ~. \\
  1 &= |\{f\in B_n\mid f(n)=1\}|~.
\end{align*}
Similarly, for part (5), we divide the elements of $B_{p, q}$ into two groups: the elements whose ranges contain $n-p$, and those whose ranges do not contain $n-p$. Part (5) follows from the following three identities:
\begin{align*}
    b_{p-1,\,q}&=|\{f\in B_{p,q}\mid n-p\notin R(f)\}|~.\\
    b_{p-1,\,r}&=|\{f\in B_{p,q}\mid f(n-1-r)=n-p\}|\quad\text{for}\quad 0\leq r\leq q-1~.\\
    1 &=|\{f\in B_{p,q}\mid f(n)=n-p\}|~.
\end{align*}
\end{proof}

\section{Modules for $B_n$}
A vector space $V$ over a field $F$ of characteristic $0$ is called a $B_n$-module if $B_n$ acts on $V$ satisfying, for all $f, f_1, f_2\in B_n$, $u, v\in V$, and $\lambda\in F$,
\begin{eqnarray*}
  f\cdot (u+v) &= f\cdot u + f\cdot v, \quad\quad\quad f_1\cdot (f_2\cdot u) &= (f_1f_2)\cdot u, \\
  f\cdot (\lambda u)  &= \lambda (f\cdot u),~\quad\quad\quad\quad\quad\quad\quad 1\cdot u &= u.
\end{eqnarray*}

From now on, $V$ denotes a vector space with a basis
$
   \mathcal{B} = \{v_S\mid S\subseteq\n\}
$
indexed by all the subsets of $\n$. Then $V=\bigoplus_{S\subseteq\n} Fv_S$ as subspaces is a $B_n$-module with respect to the following action: for $f\in B_n$ and $S\subseteq\n$,
\begin{equation}\label{moddef}
f\cdot v_S=
\left\{
  \begin{array}{ll}
    v_{S^\p}, & \hbox{if\; $S\subseteq D(f)$} \\
    0, & \hbox{otherwise,}
  \end{array}
\right.
\end{equation}
where $S^\p=\{f(s_1),\dots,f(s_k)\}$ if $S=\{s_1,\dots,s_k\}$. For $0\le k \le n$, let
$$
    V_k=\mathrm{span}\{v_S\in\mathcal{B}\mid k=|S|\}.
$$
Then $V=\bigoplus^n_{k=0}V_k$ is a direct sum of $B_n$-submodules.

Every module under consideration is a $B_n$-module over $F$, unless otherwise stated.
Our intent below is to describe the $B_n$-module structure of $V_k$ and $V$. To this end, we
define a partial order on the power set of $\n$. For any $k$-subsets
$S=\{s_1<\dots<s_k\}$ and $T=\{t_1<\dots<t_k\}$ of $\n$, define
\[
    T\leq S \quad \Leftrightarrow \quad t_i\leq s_i \quad\text{for all}\quad i\in \mathbf{k}~,
\]
and a $k$-subset is not comparable to any $l$-subset if $k\ne l$.

For $v\in V$ we use $B_nv$ to denote the cyclic submodule of $V$ generated by $v$. If $S$ is a $k$-subset of $\n$, then $B_nv_S$ is a submodule of $V_k$. Indeed, for any $f\in B_n$ if $S\subseteq D(f)$ then $f(S)$ is a $k$-subset, so $f\cdot v_S = v_{f(S)}\in V_k$; if $S$ is not a subset of $D(f)$ then $f\cdot v_S = 0\in V_k$.  Some further properties of the module $B_nv_S$ are described in the next result.
\begin{lemma}\label{mod1} Let $S,T$ be $k$-subsets of $\n$.

{\rm(1)} $B_nv_S = \bigoplus_{S'\subseteq\n,\, S'\leq S}Fv_{S'}$ as vector spaces. In particular, $V_k=B_nv_{\{n-k+1,\, \ldots,\, n\}}$.

{\rm(2)} $B_nv_T\subseteq B_nv_S$ if and only if $T\leq S$.

{\rm(3)} $B_nv_S\cap B_nv_T = B_nv_{S\wedge T}$, where $S\wedge T$ is the greatest lower bound of $S$ and $T$.
\end{lemma}
\begin{proof}

To prove (1) notice that two subsets $S'\leq S$ if and only if $S=D(f)$ and $S'=R(f)$ for a unique $f\in B_n$. Let $S'\le S$. Then $v_{S'} = f\cdot v_S\in B_nv_S$. Hence $\bigoplus_{S'\subseteq\n,\, S'\leq S}Fv_{S'}$ is included in $B_nv_S$. Conversely, let $x=g\cdot v_S \ne 0$ for some $g\in B_n$. We have $S\subseteq D(g)$, $g(S)\le S$, and hence $x=v_{g(S)}\in \bigoplus_{S'\subseteq\n,\, S'\leq S}Fv_{S'}$. The second part of (i) is now clear.

The proof of (2) follows from (1) since $\{T'\mid T'\subseteq\n,\, T'\leq T\}\subseteq\{S'\mid S'\subseteq\n,\, S'\leq S\}$ if and only if $T\le S$.

To prove (3) let $g\cdot v_S = h\cdot v_T\ne 0$ for some $g, h\in B_n$. Then $g(S)=h(T)$. Suppose
\[
    S = \{s_1< \ldots < s_k\}\quad\text{and}\quad T = \{t_1< \ldots < t_k\}~.
\]
Then
$
    S \wedge T =\{\min(s_1, t_1),\, \ldots,\, \min(s_k, t_k)\},
$
and $g(s_i)=h(t_i)$. We define $f\in B_n$ with $D(f)=S \wedge T $ and $R(f)=g(S)$ by
$
    f(\min(s_i,\, t_i)) =  g(s_i),
$
where $1\le i\le k$. Then $g\cdot v_S = f\cdot v_{S\wedge T} \in B_n v_{S\wedge T}$, and hence $B_nv_S\cap B_nv_T \subseteq B_nv_{S\wedge T}$.
Conversely, for any given $0\ne f\cdot v_{S\wedge T} \in B_n v_{S\wedge T}$ define $g(s_i)=h(t_i)= f(\min(s_i,\, t_i))$ for $1\le i\le k$. Then $f\cdot v_{S\wedge T} = g\cdot v_S=h\cdot v_T\in B_nv_S\cap B_nv_T$. The proof of (3) is complete.
\end{proof}

    Let $v=\sum_{S\subseteq \n}\lambda_Sv_S, \lambda_S\in F$ be a vector of $V$. The {\em support} of $v$ is defined to be
    \[
        {\rm supp}(v) = \{S\subseteq\n\mid \lambda_S \ne 0\}~.
    \]
%
\begin{definition}\label{redGen}
    A vector of the form $w=\sum_{S\in {\rm supp}(w)}v_S\in V$ is called a {\em reduced generator} of a submodule $W$ of $V$ if $W=B_nw$ and $W$ cannot be generated by any other vector whose support contains fewer elements than {\rm supp(}$w${\rm )}. We agree that $0$ is the reduced generator of the zero submodule.
\end{definition}
The next proposition gives some properties of submodules of $V$.
\begin{proposition}\label{cyclic} Let $v=\sum_{S\in\,{\rm supp(}v{\rm )}}\lambda_Sv_S\in V$.

{\rm (1)} If $S$ is in {\rm supp}$(v)$, then $v_S\in B_nv$~.

{\rm (2)} $B_nv = \bigoplus_{T\in \mathcal{P}(v)} Fv_T$ as subspaces, where $\mathcal{P}(v) = \bigcup_{S\in {\rm supp}(v)}\{T\subseteq\n\mid T\le S\}$.

{\rm (3)} Every submodule of $V$ is cyclic and contains a unique reduced generator.
\end{proposition}
\begin{proof} To prove (1) let $\min \big\{\,|S| \,\big|\, S \in {\rm supp}(v)\big\}=r$. Then there exists an $r$-subset $T=\{t_1<\cdots < t_r\}\subseteq\n$ such that $T\in {\rm supp}(v)$; if $r=0$, then $T=\emptyset$. Let $f\in B_n$ such that $D(f)=R(f)=T$. By the choice of $r$, for every $S\in{\rm supp}(v)$ with $S\neq T$, there is at least one $s\in S$ such that $s\notin T$, so $f\cdot v_S=0$. Hence
$$
    f\cdot v=f\cdot \sum_{S\in\,{\rm supp(}v{\rm )}}\lambda_Sv_S=\sum_{S\in\,{\rm supp(}v{\rm )}}\lambda_S(f\cdot v_S)=\lambda_Tv_T~.
$$
Thus $v_T\in B_nv$ since $\lambda_T\neq 0$. It is easily seen that
$$
    \sum_{S\in\,{\rm supp(}v{\rm )}\atop |S|>r}\lambda_Sv_S=v-\sum_{S\in\,{\rm supp(}v{\rm )}\atop |S|=r}\lambda_Sv_S\in B_nv~.
$$
Applying the above procedure to
$
    \sum_{S\in\,{\rm supp(}v{\rm )},\,|S|>r}\lambda_Sv_S
$
and iteratively using this procedure, if needed, we get $v_S\subseteq B_nv$ for all $S\in \text{supp}(v)$. The proof of (1) is complete.

From (1) and Lemma \ref{mod1} (1), we have
\begin{eqnarray*}
   B_nv&=& \sum_{S \in\text{\rm supp}(v)}\lambda_SB_nv_S \\
&=&\sum_{S \in\text{\rm supp}(v)}\mathrm{span}\,\{v_T\in\mathcal{B}\mid T\leq S\}\\
&=& \bigoplus_{S\in \mathcal{P}(v)} Fv_S, \quad\text{as subspaces}.
\end{eqnarray*}
This completes the proof of (2).

We now prove (3). It is trivial for $W=\{0\}$. Let W be a nonzero submodule of $V$. We claim that $W$ has a basis $\{v_S\in\mathcal{B}\mid S\in \mathcal{P}\}$ for some subset $\mathcal{P}$ of the power set of $\n$. Indeed, suppose $\mathcal{B}_1$ is a basis of $W$ and write every element of $\mathcal B_1$ as a linear combination of basis vectors in $\mathcal{B}=\{v_S\mid S\subseteq \n\}$. Let $\mathcal P$ be the set of all the different subsets $S$ where $S$ runs through the support of every element of $\mathcal B_1$. By (1) the set $\{v_S\in\mathcal{B}\mid S\in \mathcal{P}\}$ is a subset of $W$, and hence a basis of $W$ since it is linearly independent and spans $W$. Let
$
    w=\sum_{S\in\mathcal{P}} v_S.
$
By (1) again, $W$ is generated by $w$, and hence $W$ is cyclic.

We now show how to deduce a reduced generator of $W$ from $w$. Indeed, if $w$ contains two vectors $v_S, \,v_T$ with $T\le S$ and $T\ne S$ in supp($w$), then we can remove the term $v_T$ from $w$, and by Lemma \ref{mod1} (i) the sum of the remaining terms is still a generator. Repeat this process until we obtain the set
\[
    {\rm Red}(w) = \{S\mid  S \;\text{is maximal in supp}(w)\},
\]
and then we define the corresponding generator $w_{\rm red}$ of $W$ by
\[
    w_{\rm red}=\sum_{S\in{\rm Red}(w)}v_S~.
\]
We claim that $w_{\rm red}$ is a reduced generator of $W$. Let $v = \sum_{S\in\text{supp}(v)}\lambda_Sv_S$ be another generator of $W$. From Definition \ref{redGen} it suffices to show that $|{\rm supp(}v{\rm )}| \ge |{\rm Red(}w{\rm )}|$.
From (2) we find $W = \bigoplus_{T\in \mathcal{P}(v)} Fv_T = \bigoplus_{T\in \mathcal{P}(w)} Fv_T$ where $\mathcal{P}(v)$ and $\mathcal{P}(w)$ are as in (2),
and hence
$
    \mathcal{P}(v) = \mathcal{P}(w).
$
Define
\begin{equation}\label{redv}
    {\rm Red}(v) = \{S\mid  S \text{ is maximal in supp}(v)\}.
\end{equation}
Thus, ${\rm Red}(v) = \{S\mid S \text{ is maximal in }\mathcal{P}(v)\}$ and
${\rm Red}(w) = \{S\mid  S \text{ is maximal in }\mathcal{P}(w)\}$.
So, Red($v$) = Red($w$) and $|{\rm supp(}v{\rm )}| \ge |{\rm Red(}v{\rm )}| = |{\rm Red(}w{\rm )}|$, showing that $w_{\rm red}$ is reduced.

Suppose that $v = \sum_{S\in{\rm supp}(v)}v_S$ is another reduced generator of $W$. By the definition of reduced generators we know $|{\rm supp(}v{\rm )}| = |{\rm Red(}w{\rm )}|$. Hence $|{\rm supp(}v{\rm )}| =  |{\rm Red(}v{\rm )}|$ since Red($v$) = Red($w$). It follows that ${\rm supp(}v{\rm )} = {\rm Red(}v{\rm )}$. Let $v_{\rm red}=\sum_{S\in{\rm Red}(v)}v_S$. Then $v=v_{\rm red}=w_{\rm red}$. Therefore $w_{\rm red}$ is the unique reduced generator of $W$.
\end{proof}

\begin{definition} The set ${\rm Red}(v)$ in {\rm (\ref{redv})} is called the {\em reduced support} of $v$, and the element $v_{\rm red}=\sum_{S\in{\rm Red}(v)}v_S$ is termed the {\em reduced form} of $v$. The reduced support of $0$ is empty, and the reduced form of $0$ is itself.
\end{definition}

For example, if $n=7$ and $v = v_\emptyset - 2v_{\{1\}} + v_{\{3\}} + 5 v_{\{1, \,2\}} + 3v_{\{4,\, 7\}} - 2v_{\{5, \,6\}} + v_{\{1, \,2, \,3\}}$, then Red($v$) = $\{\emptyset,\,\{3\},\,\{5,\, 6\},\,\{4,\, 7\}, \{1, \,2, \,3\}\}$ is the reduced support of $v$, and its reduced form is $v_{\rm red} = v_\emptyset + v_{\{3\}} + v_{\{4,\, 7\}} + v_{\{5, \,6\}} + v_{\{1, \,2, \,3\}}$.

It is sometimes convenient to call the reduced support of $v$ the {\em reduced support} of the module $B_nv$.
A direct calculation yields that the reduced generator of $V_k$ is $v_{\{n-k+1,\,\ldots,\,n\}}$ for $1\le k\le n$, and the reduced support of $V_k$ is the set $\{n-k+1,\,\ldots,\,n\}$. The module $V_0$ has the element $v_\emptyset$ as its reduced generator, and its reduced support is the set $\{\emptyset\}$.

The next result is a consequence of Lemma \ref{mod1} (i) and Proposition \ref{cyclic} (3).

\begin{corollary}\label{eq}
    If $v, w\in V$, then $B_n v = B_n w$ if and only if they have the same reduced support {\rm Red(}v{\rm)} =  {\rm Red(}w{\rm)} if and only if they have the same reduced generator $v_{\rm red} = w_{\rm red}$.
\end{corollary}

We can now describe the irreducible submodules of $V_k$ for $0\le k\le n$. Write $\bold k = \{1,\,\dots,\,k\}$. If $k=0$, we agree that $\bold k=\emptyset$ and $v_\bold k = v_\emptyset$.
\begin{proposition}\label{irreVk}
For each $0\le k\le n$, the 1-dimensional submodule $B_nv_{\bold k}$ is the only irreducible submodule of $V_k$, and every nonzero submodule of $V_k$ contains $B_nv_{\bold k}$.
\end{proposition}
\begin{proof} Since ${\bold k}$ is the smallest element of the set of all $k$-subsets and the elements of $B_n$ are order decreasing as well as order preserving injective maps, from action (\ref{moddef}) we find $B_nv_{{\bold k}}=Fv_{\bold k}$ is an irreducible submodule of $V_k$, and it is $1$-dimensional.
If $W$ is another nonzero irreducible submodule of $V_k$, by Proposition \ref{cyclic} (3) there exists a generator $w\in V_k$ such that $W = B_nw$. The irreducibility of $W$ forces that supp($w$) contains only the $k$-subset ${\bold k}$, since if supp($w$) contains another $k$-subset $S$ different from ${\bold k}$, then by Lemma \ref{mod1} (1), $v_S\in W\setminus B_nv_{\bold k}$ and hence $B_nv_{\bold k}$ would be a nonzero proper submodule of $W$.  We conclude $W=B_nv_{{\bold k}}$.

Now let $W$ be any nonzero submodule of $V_k$. From Proposition \ref{cyclic} (3) we know that $W$ is generated by a nonzero element $v\in V_k$. Pick any $S\in$ supp($v$). Then $v_S\in W$ by Proposition \ref{cyclic} (1). Since there exists a unique map $f\in B_n$ such that $D(f)=S$ and $R(f)={\bold k}$, we find $v_{\bold k}=f\cdot v_S\in B_nv_S\subseteq W$. Therefore $B_nv_{\bold k}\subseteq W$.
\end{proof}

The next result describes irreducible submodules of $V$ in terms of those of $V_k$.
\begin{proposition}\label{irreV}
If $W$ is an irreducible submodule of $V$, then $W=B_nv_{{\bold k}}$ for some $0\le k\le n$ and $\dim W =1.$ Moreover, $\{B_nv_{{\bold k}} \mid k = 0,\,\ldots,\, n\}$ is a complete set of irreducible submodules of $V$.
\end{proposition}

\begin{proof}
 From Proposition \ref{cyclic} (3) we find $W = B_nw$ for a reduced generator $w\in V$. Since $W$ is nonzero, supp($w$) is not empty. Assume that $S,\,T\in$ supp($w$) and $S\ne T$. Since $S,\,T$ are different maximal elements in supp($w$), from Lemma \ref{mod1} (i) we find $v_T\in W\setminus B_nv_S$, and hence $B_nv_S$ is a nonzero proper submodule of $W$, which contradicts the irreducibility of $W$. Therefore, supp($w$) contains only one subset of $\n$, showing that $W$ is a submodule of $V_k$ for some $0\le k\le n$. It follows from  Proposition \ref{irreVk} that $W=B_nv_{{\bold k}}$ and $\dim W =1$. The second part of the proposition is now straightforward.
\end{proof}

 Recall that a $B_n$-module is {\em indecomposable} if it is nonzero and cannot be written as a direct sum of two nonzero submodules, and that a $B_n$-module is called {\em completely decomposable} if it is nonzero and is a direct sum of indecomposable submodules.
\begin{proposition}\label{indDecom}
Let $W$ be a nonzero submodule of $V$. Then $W$ is indecomposable if and only if $W$ is a submodule of some $V_k$ where $0\le k\le n$.
\end{proposition}
\begin{proof}
    If $W$ is a nonzero submodule of $V_k$ where $0\le k\le n$, then by Proposition \ref{irreVk} any two nonzero submodules of $W$ both contain $B_nv_{{\bold k}}$, so their sum cannot be direct, and hence $W$ is indecomposable.

    Conversely, if $W$ is a nonzero indecomposable submodule of $V$, from Proposition \ref{cyclic} (3) it follows that $W=B_nv$ for a unique reduced generator $v = \sum_{S\in {\rm Red}(v)}v_S\in V$. Let $\mathcal{P}(i)$ be the set of all $i$-subsets of $\n$ where $0\le i\le n$. For each $i$ let Red$_i(v) = {\rm Red(}v{\rm)}\cap \mathcal{P}(i)$. Forgetting all the possible empty Red$_i(v)$, we obtain a partition of
    \[
        {\rm Red}(v) = {\rm Red}_{i_1}(v) \sqcup \cdots \sqcup {\rm Red}_{i_s}(v), \quad\text{for some } 1\le s\le n+1,
    \]
    where $0\le i_1 < \cdots < i_s \le n.$ Let $v_{i_j} = \sum_{S\in {\rm Red}_{i_j}(v)} v_S$ be the reduced vector with support Red$_{i_j}(v)$ where $1\le j\le s$. Then
    \begin{equation}\label{comDom}
      W = \bigoplus_{j=1}^{s} B_nv_{i_j}~,  \quad\text{direst sum of submodules}.
    \end{equation}
    Since $W$ is indecomposable and each $B_nv_{i_j}$ is a nonzero proper submodule of $W$, there exists some $k=i_j$ such that $W = B_nv_k$, which is a submodule of $V_k$.
\end{proof}

\begin{corollary}
    Every nonzero submodule of $V$ is completely decomposable. In particular, $V$ is completely decomposable and $V=\bigoplus_{k=0}^{n}V_k$ is a direct sum of indecomposable submodules.
\end{corollary}
\begin{proof}
     Let $W$ be a nonzero submodule of $V$. Then  from Proposition \ref{cyclic} (3) we have $W=B_nv$ for a unique reduced generator $v = \sum_{S\in {\rm Red}(v)}v_S\in V$. A similar argument to that of Proposition \ref{indDecom} shows that $W$ has the decomposition (\ref{comDom}). From Proposition \ref{indDecom} each $B_nv_{i_j}$ in (\ref{comDom}) is indecomposable. The first part of the desired result follows. The second part follows immediately.
\end{proof}

\begin{proposition}
No two different submodules of $V$ are isomorphic.
\end{proposition}
\begin{proof}
Let $W$ and $U$ be two submodules of $V$ and $\sigma:W\rightarrow U$ be a module isomorphism. Let $x=\sum_{S\in\, \mathcal I} v_S$ be a reduced generator of $W$, where $\mathcal I$ is an index set. Now for $v_S,S\in \mathcal I$, suppose $\sigma(v_S)=\sum_{T\in\, \mathcal J} \lambda_T v_T$ for some index set $\mathcal J$, where $\lambda_T\in F$. Take $f_S\in B_n$ with $D(f)=R(f)=S$. Then
\[
    \sigma(v_S)=\sigma(f_S\cdot v_S)=f_S\cdot\sigma(v_S)=f_S\cdot\sum_{T\in \mathcal J} \lambda_T
v_T=\sum_{T\in \mathcal J,\,T\subseteq S} \lambda_T (f_S\cdot v_T)=\sum_{T\in
\mathcal J,\,T\subseteq S} \lambda_T v_T.
\]
We show that $\sigma(v_S)=\lambda_Sv_S$ with $\lambda_S\neq 0$. If in the sum on the right there is some $T^\p\in \mathcal J,\,T^\p\subsetneq S$ with $\lambda_{T^\p}\neq 0$, let $f_{T^\p}\in B_n$ with $D(f_{T^\p})=R(f_{T^\p})=T^\p$. We have
\[
    f_{T^\p}\cdot\sigma(v_S) = f_{T^\p}\cdot\sum_{T\in \mathcal J,\,T\subseteq S}\lambda_T v_T=\sum_{T\in \mathcal J,\,T\subseteq S}\lambda_T (f_{T^\p}\cdot v_T)=\sum_{T\in \mathcal J,\,T\subseteq T^\p}\lambda_T v_T\ne 0,
\]
but $f_{T^\p}\cdot\sigma(v_S) = \sigma(f_{T^\p}\cdot v_S) =\sigma(0)=0$, a contradiction. Thus we get $\sigma(v_S)=\lambda_Sv_S$ and $\lambda_S\neq 0$, so $B_nv_S=B_n\sigma(v_S)$. By Proposition 4.3 (1), we find
\[
W=B_nx=\sum_{S\in \mathcal I} B_nv_S=\sum_{S\in \mathcal I}B_n\sigma(v_S)=\sigma\Big(\sum_{S\in \mathcal I}B_nv_S\Big)=\sigma\Big(B_n\big(\sum_{S\in \mathcal I}v_S\big)\Big)=\sigma(B_nx)=U.
\] \end{proof}

We now describe the dimension of any nonzero submodule of $V$. Proposition \ref{cyclic} (3) assures that the submodule is equal to the module $B_nv$ generated by some $v\in V$.
\begin{proposition}
Let $v\in V$ and {\rm Red(}v{\rm )}= $\{S_1,\,\ldots,\,S_m\}$. For any $J\subseteq$ {\rm Red(}v{\rm )} denote by $S_J$ the greatest lower bound of $\{S_j\mid j\in J\}$. Then the dimension of $B_nv$ is given by
\[
    \dim B_nv = \sum_{\emptyset\,\ne J\,\subseteq {\bold m}} (-1)^{|J|-1} \dim B_nv_{S_J}.
\]
\end{proposition}
\begin{proof}
  From Proposition \ref{cyclic} (2) and (3) the dimension of $B_nv$ is equal to the cardinality of the set
  $\mathcal{P}(v) = \bigcup_{S\in {\rm Red(}v{\rm )}}\{T\subseteq\n\mid T\le S\}$. Let $A_j=\{T\subseteq\n\mid T\le S_j\}$, $j\in{\bold m}$. Then $\mathcal{P}(v) = \bigcup_{j\in{{\bold m}}}A_j$, and $\dim B_nv_{S_j} = |A_j|$ by Proposition \ref{cyclic} (2). With Lemma \ref{mod1} (3) in mind and applying the inclusion-exclusion principle to count the cardinality of $\mathcal{P}(v)$, we obtain the desired formula for $\dim B_nv$.
\end{proof}

We further describe the dimension of $B_nv_S$ for any $S\subseteq\n$. In what follows we agree that if $x>y$, then $\binom{y}{x}=0$ and the empty sum $\sum_{i=x}^y\square_i = 0$.
\begin{theorem}\label{dim}
If $S=\{s_1<\dots<s_k\}$ is a $k$-subset of $\n$, let $d_k$ be the dimension of the module $B_nv_S$. We have $d_1 = s_1$, and for $k\ge 2$,
\begin{align}
   d_k &= \sum^{k-1}_{i=1}{s_{k-i+1} \choose k+1-i}\gamma_i -\sum^{k-1}_{i=1}{s_{k-i+1}-s_1 \choose k+1-i}\gamma_i-\sum^{k-2}_{i=1}s_1 {s_{k-i+1}-s_2 \choose k-i}\gamma_i ~, \label{dk}
\end{align}
where $\gamma_1=1$ and for $2\leq j\leq k-1$,
\begin{equation*}\label{}
    \gamma_j=-\sum^{j-2}_{i=1}
              {s_{k+1-i}-s_{k+2-j} \choose j-i}\gamma_i ~.
    \end{equation*}
\end{theorem}
\begin{proof}
By Lemma \ref{mod1} (i) we know that $d_k$ is equal to the number of $k$-subsets $T$ of $\n$ such that $T\leq S$. Let
$
    \la_i=s_{k-i+1}-(k-i+1)\text{ for } 1\le i \le k.
$
Then $\lambda_i\ge\lambda_{i+1}$ since $s_{k-i+1}>s_{k-i}$.
Because the smallest $k$-subset is $\{1,\dots,k\}$, we have
\begin{equation}\label{lambdaSequence}
    \la_1\geq \dots\geq\la_k\geq 0,
\end{equation}
and the number of $k$-subsets $T$ of $\n$ with $T\leq S$ is equal to the number of all the sequences
\begin{equation}\label{partition}
    \mu_1\geq \dots\geq \mu_k\geq 0\quad\text{with}\quad \mu_i\leq \lambda_i\quad\text{for}\quad i=1,\dots,k.
\end{equation}
To find $d_k$ it suffices to compute the number of the sequences in (\ref{partition}) for the given sequence (\ref{lambdaSequence}).
If $k=1$, then $d_1 = \lambda_1 + 1 = s_1$.

If $k\ge 2$, let $2\le j\le k$. For each fixed nonnegative integer $\mu\le\la_j$ we calculate iteratively on $j$ the number $\alpha_j(\mu)$ of the sequences
\begin{equation}\label{aSj}
  \mu_1\geq \dots\geq \mu_{j-1}\ge\mu\quad\text{with}\quad \mu_i\leq \lambda_i\quad\text{for}\quad i=1,\dots,j-1,
\end{equation}
and the required dimension $d_k = \sum_{\mu = 0}^{\lambda_k}\alpha_k(\mu)$.

Let $\xi_j=\la_j-\mu$. Then $0\leq \xi_j\leq\la_j$. Our aim now is to prove
\begin{equation}\label{alphaj}
  \alpha_j(\mu) = \beta_j + \gamma_j,
\end{equation}
where
$\beta_j=\sum^{j-1}_{i=1}{\la_i-\la_j+\xi_j+j-i \choose j-i}\gamma_i$ and $\gamma_j = -\sum^{j-2}_{i=1}{\la_i-\la_{j-1}+j-i-1 \choose j-i}\gamma_i$ with $\gamma_1=1$.
Notice that $\alpha_j(\mu)$ is a sum of two numbers $\beta_j$ and $\gamma_j$, of which $\gamma_j$ depends on $\la_1,\dots,\la_{j-1}$, whereas $\beta_j$ depends on $\la_1,\dots,\la_j$ and $\xi_j$.

We use induction on $j$ to prove (\ref{alphaj}) for $2 \le j\le k$. If $j=2$, for each fixed nonnegative integer $\mu\le\lambda_2$
we have $\xi_2=\la_2-\mu$ and $0\leq \xi_2\leq\la_2$. Let $\xi_1=\la_1-\mu_1$. To ensure that (\ref{aSj}) holds for this case, namely $\mu_1\geq\mu$ and $\mu_1\leq\lambda_1$, we must have $0\leq \xi_1\leq\la_1-\la_2+\xi_2$, and conversely. So
\begin{equation*}
    \alpha_2(\mu)= \la_1-\la_2+\xi_2+1 = \beta_2+\gamma_2
\end{equation*}
where $\beta_2=\la_1-\la_2+\xi_2+1$ and $\gamma_2=0$, and this is (\ref{alphaj}) for $j=2$.

Suppose (\ref{alphaj}) holds for $j=l$ with $2\leq l\leq k-1$, that is, for each fixed nonnegative integer $\mu\le\lambda_l$
we have $\xi_l=\la_l-\mu$ with $0\leq \xi_l\leq\la_l$, and the number of sequences $\mu_1\geq\dots\geq\mu_{l-1}\geq\mu$ with $\mu_i\leq\lambda_i$ for $i=1,\ldots,l-1$ is
\begin{equation}\label{hypothesis}
  \alpha_l(\mu) = \beta_l + \gamma_l,
\end{equation}
where
$\beta_l=\sum^{l-1}_{i=1}{\la_i-\la_l+\xi_l+l-i \choose l-i}\gamma_i$ and $\gamma_l = -\sum^{l-2}_{i=1}{\la_i-\la_{l-1}+l-i-1 \choose l-i}\gamma_i$.

We now prove (\ref{alphaj}) for $j=l+1$. For a fixed nonnegative integer $\nu\le\lambda_{l+1}$ we have $\xi_{l+1} = \la_{l+1}-\nu$ with $0\leq \xi_{l+1}\leq\la_{l+1}$. Let $\mu=\la_l-\xi_l$. To ensure that the condition (\ref{aSj})
\[
     \mu_1\geq\dots\geq\mu_{l-1}\geq\mu\ge\nu\quad\text{ with}\quad\mu_i\le\lambda_i,\, i=1,\ldots,{l-1}\text{ and } \mu\le\lambda_l
\]
holds here, we must have $0\leq \xi_l\leq\rho_l$ where $\rho_l=\la_l-\la_{l+1}+\xi_{l+1}$, and conversely. Adding all $\alpha_l(\mu)$ up for $\nu\le\mu\le\lambda_l$ and using the induction hypothesis (\ref{hypothesis}), we obtain
\begin{eqnarray}
\nonumber\alpha_{l+1}(\nu)&=&\sum_{\mu = \nu}^{\lambda_{l}} \alpha_l(\mu)  = \sum_{\mu = \nu}^{\lambda_{l}} (\beta_l + \gamma_l)\\
          &=&\sum^{\rho_l}_{\xi_l=0}\sum^{l-1}_{i=1}{\la_i-\la_l+\xi_l+l-i \choose l-i}\gamma_i+\sum^{\rho_l}_{\xi_l=0}\gamma_l \label{second}\\
\nonumber &=&\sum^{l-1}_{i=1}\left\{{\la_i-\la_{l+1}+\xi_{l+1}+(l+1)-i \choose l+1-i}\gamma_i
    -{\la_i-\la_l+l-i \choose l+1-i}\gamma_i\right\}\\
          &&\qquad\qquad\qquad\qquad\qquad+{\la_l-\la_{l+1}+\xi_{l+1}+1\choose 1}\gamma_l \label{alp1}\\
\nonumber &=&\sum^l_{i=1}{\la_i-\la_{l+1}+\xi_{l+1}+(l+1)-i \choose l+1-i}\gamma_i-\sum^{l-1}_{i=1}{\la_i-\la_l+l-i \choose l+1-i}\gamma_i\\
\nonumber &=&\beta_{l+1} + \gamma_{l+1},
\end{eqnarray}
where
\begin{eqnarray*}\label{}
\beta_{l+1}&=&\sum^l_{i=1}
    {\la_i-\la_{l+1}+\xi_{l+1}+(l+1)-i \choose l+1-i}\gamma_i ~,\\
\gamma_{l+1}&=&-\sum^{l-1}_{i=1}{\la_i-\la_l+l-i \choose
l+1-i}\gamma_i ~.
\end{eqnarray*}
Here we have made use of the identity $\sum_{z=a}^{a+b-1}\binom{z}{p} = \binom{a+b}{p+1} - \binom{a}{p+1}$ in which $a, b, p$ are natural numbers to obtain (\ref{alp1}) from (\ref{second}) by assigning $a=\lambda_i - \lambda_l + l - i,\, b=\lambda_l - \lambda_{l+1} + \xi_{l+1} + 1$ and $p=l-i\ge 1$.
Therefore, (\ref{alphaj}) is valid for $j=l+1$, and we complete the proof of (\ref{alphaj}) by induction.

We are now able to calculate the dimension $d_k$ of the module $B_nv_S$ for $k\ge 2$ by summing all $\alpha_k(\mu)$ in (\ref{alphaj}) up where $\mu$ runs from $0$ to $\la_k$, yielding
\begin{eqnarray*}\label{dimlambda3}
\nonumber d_k &=& \sum^{\la_k}_{\mu=0}\alpha_k(\mu) \\
\nonumber   &=&\sum^{\la_k}_{\xi_k=0}\sum^{k-1}_{i=1}
                {\la_i-\la_k+\xi_k+k-i \choose k-i}\gamma_i-\sum^{\la_k}_{\xi_k=0}\sum^{k-2}_{i=1}
                {\la_i-\la_{k-1}+k-i-1 \choose k-i}\gamma_i\\
\nonumber   &=&\sum^{k-1}_{i=1}{\la_i+k-i+1 \choose k+1-i}\gamma_i -\sum^{k-1}_{i=1}{\la_i-\la_k+k-i \choose k+1-i}\gamma_i \\
            &&\qquad\qquad -\sum^{k-2}_{i=1}(\la_k+1) {\la_i-\la_{k-1}+k-i-1 \choose k-i}\gamma_i~.
\end{eqnarray*}
From $\la_i=s_{k-i+1}-(k-i+1)$, we conclude that $d_k$ is given by (\ref{dk}).
\end{proof}

The following four corollaries are consequences of Theorem \ref{dim}. Considering the sequence $\la_1\geq \dots\geq\la_k\geq 0$ given in (\ref{lambdaSequence}) a partition of $d=\sum_{i=1}^k\la_i$ into at most $k$-parts, we obtain the next result.
\begin{corollary}
Let $\la_1\geq\dots\geq\la_k\geq 0$ be a given partition
of some nonnegative integer. Then the number of distinct Young diagrams obtained from the Young diagram of $\la$ by removing zero or more boxes from the rows is $d_k$, which is given in Theorem {\rm \ref{dim}}.
\end{corollary}
\begin{proof}
  It is easily seen that the number of distinct Young diagrams obtained from the Young diagram of $\la$ by removing zero or more boxes from the rows is equal to the number of partitions $\mu_1\geq\dots\geq\mu_k\geq 0$ such that $\mu_i\leq \la_i$ for all $i=1,\dots,k$.  The desired result follows from the proof of Theorem \ref{dim}.
\end{proof}

\begin{corollary}\label{dimspeci1}
If $S=\{2,4,\ldots,2k\} \subseteq \mathbf{n}$, then the dimension of the submodule $B_nv_S$ is the Catalan number
$c_{k+1}$.
\end{corollary}
\begin{proof}
           The sequence $\la_1\geq \dots\geq\la_k\geq 0$ in (\ref{lambdaSequence}) associated to $S$ is now $k \geq k-1\geq\cdots\geq 1$. In this case, it is well known that the number of sequences $\mu_1\ge\mu_2\ge\ldots\ge\mu_k\ge 0$ such that $\mu_i\leq\la_i$ for $1\leq i\leq k$ is the Catalan number $c_{k+1}$. The desired result follows from Theorem \ref{dim}.
\end{proof}

We find a combinatorial identity below for the Catalan number. To our knowledge, the identity is new.
\begin{corollary}\label{comId}
If $k\geq 2$, then
\begin{equation}\label{catalannewform1}
c_{k+1}=\sum^{k-1}_{i=1}{2(k-i+1) \choose k+1-i}\gamma_i
  -\sum^{k-1}_{i=1}{2(k-i) \choose k+1-i}\gamma_i-\sum^{k-2}_{i=1}2
          {2(k-i-1) \choose k-i}\gamma_i~,
\end{equation}
where $\gamma_1=1$ and for $2\leq i\leq k$,
\begin{equation}\label{catalannewform2}
\gamma_i =-\sum^{i-2}_{j=1}
          {2(i-j-1) \choose i-j}\gamma_j~.
\end{equation}
\end{corollary}
\begin{proof}
  The  combinatorial identity is obtained from Theorem \ref{dim} and Corollary \ref{dimspeci1}.
\end{proof}

\begin{corollary}\label{dimspeci2}
For the $k$-subset $\{m+1,\ldots,m+k\}$ of $\mathbf{n}$, the
dimension of the submodule $B_nv_S$ is ${m+k \choose k}$.
\end{corollary}
\begin{proof}
   The sequence in (\ref{lambdaSequence}) corresponding to $S$ is the $k$-subset $\{m, \ldots,m\}$. A direct calculation of $d_k$ for $k\geq 1$ using the formulas given in Theorem \ref{dim} yields
    \begin{equation}\label{}
        d_k={m+k \choose k}~,
    \end{equation}
which is the desired result.
\end{proof}

We now compute the dimension $d_{k,\, m}$ of the $B_n$-module $B_nv_{S_{k,\,m}}$, where $k\ge m$ and
$$
    S_{k,\,m}=\{2,\,4,\,\ldots,\,2m,\,2m+1,\,2m+2,\,\ldots,\,2m+(k-m)\}
$$
is a subset of $\mathbf{n}$, which is a mixture of the two types of subsets in Corollaries \ref{dimspeci1} and \ref{dimspeci2}. The sequence (\ref{lambdaSequence}) associated to $S_{k,\,m}$ is $\{m,\, m,\,\ldots,m,\,m-1,\,m-2,\,\ldots,\,1\}$ of length $k$.
Recall that if $x>y$, then $\binom{y}{x}=0$ and the empty sum $\sum_{i=x}^y\square_i = 0$. Without showing the details, from Theorem \ref{dim} we obtain, for $m\ge 2$,
\begin{eqnarray}\label{catagene2}
 \nonumber d_{k,\,m}&=&{m+k \choose k}-{m+k-2 \choose k}-2
          {m+k-4 \choose k-1} +\sum^{k-1}_{i=k-m+3}{2(k-i+1) \choose k+1-i}\gamma_i \\
  &&{}-\sum^{k-1}_{i=k-m+3}{2(k-i) \choose k+1-i}\gamma_i-\sum^{k-2}_{i=k-m+3}2
          {2(k-i-1) \choose k-i}\gamma_i
\end{eqnarray}
where $\gamma_1=1,\gamma_2=\gamma_3=\dots=\gamma_{k-m+2}=0,
\gamma_{k-m+3}=-1$, and for $i\geq k-m+4$

\begin{equation}\label{catagene3}
\gamma_i =-{m-k+2i-4\choose i-1}-\sum^{i-2}_{j=k-m+3}
          {2(i-j-1) \choose i-j}\gamma_j~.
\end{equation}Notice that $d_{k,\,k}$ is just the Catalan number $c_{k+1}$ by Corollary (\ref{dimspeci1}). Formula (\ref{catagene2}) will be used in Corollary \ref{ccd}.

Identifying an element of $B_t$ $(t< n)$ with an element of $B_n$ that fixes $t+1, \ldots,n$, we can regard $B_t$ as
a submonoid of $B_n$, so we are allowed to view any $B_n$-submodules of $V$, for example $V_k$ and its submodules, as $B_t$-modules.

Our aim below is to investigate decompositions of the $B_n$-module $W^m_k =B_nv_{S^m_k}$ into indecomposable submodules for $S^m_k=\{m+1,\ldots,m+k\}$ where $k$ and $m$ are positive integers and $k+m\le n$. Using the notation above, we obtain the following result.
\begin{proposition}\label{decomposition}
Viewed as a $B_{m+l}\,(1\leq l<k)$ module, $W^m_k =B_nv_{S^m_k}$ is decomposed into a direct sum of indecomposable submodules
\vspace {-2mm}
\begin{equation}\label{moddecom1}
W^m_k \downarrow_{B_{m+l}}^{B_{m+k}} \,\cong\, \bigoplus_{a=0}^{k-l} {k-l \choose a}W_{k-a}^{m+l-k+a} ~,
\end{equation}
\vspace {-2mm}
where ${k-l \choose a}$ is the multiplicity of the indecomposable
submodule $W_{k-a}^{m+l-k+a}$.
\end{proposition}
\begin{proof}
By Lemma \ref{mod1} (1) we have
\[
    W^m_k  = B_{m+k}v_{S^m_k} = \bigoplus_{k\text{-subsets } T\subseteq \{1,\, ...,\, m+k\}} Fv_T~.
\]
To obtain the desired indecomposable $B_{m+l}$-submodules on the right of (\ref{moddecom1}), we group the $1$-dimensional subspaces $Fv_T$ with $k$-subsets $T\subseteq\{1,\, ...,\, m+k\}$ into categories according to the intersection $\{t'_1, \ldots, t'_a\} = T\cap \{m+l+1,\,\ldots,\, m+k\}$ where $t'_1 < \cdots < t'_a$,\, $0\leq a\leq k-l$, and $1\leq l<k$.
Let $\mathcal{T}$ denote the set of all these $k$-subsets $T$. For any $T\in \mathcal{T}$, we have
\[
    T=\{t_1,\,\ldots,\,t_{k-a}, \,t'_1, \ldots, t'_a\}~,
\]
for some subset $\{t_1,\,\ldots,\,t_{k-a}\}\subseteq\{1, ..., m+l\}$ with $t_1 < \cdots < t_{k-a}$, so $m+l\ge k-a.$

Write $p = m+l-(k-a)$ and let $T_a = \{p+1,\,\ldots,\, m+l, \,t'_1, \ldots, t'_a\}.$ Then $T_a$ is a $k$-subset of $\{1,\ldots,m+k\}$, and
$T\le T_a$. Define $f\in B_{m+k}$ by
\[
    f(p+1)=t_1,\, \ldots,\, f(m+l) = t_{k-a},\, f(m+l+1) = m+l+1,\, \ldots,\, f(m+k) = m+k.
\]
Then $T=f(T_a)$. Identifying $f$ with an element of $B_{m+l}$, we have $v_{T}=f\cdot v_{T_a}\in B_{m+l}v_{T_a}$. From Lemma \ref{mod1} (i) we have
\[
    B_{m+l}v_{T_a} = \bigoplus_{T\in\, \mathcal{T}}Fv_T ~.
\]

We next show that
$
    B_{m+l}v_{T_a} \cong B_{m+l}v_{S_{k-a}^{p}}
$
as $B_{m+l}$-modules. Notice that
\[
    S_{k-a}^{p}=\{p+1,\,\ldots,\, p+(k-a)\} = \{p+1,\,\ldots,\, m+l\} \subseteq \{1, ..., m+l\}~.
\]
Let
$
    \mathcal{U} = \{U\subseteq \{1, ..., m+l\}\mid U\le S_{k-a}^{p}\}.
$
By Lemma \ref{mod1} (i) we have
\[
    B_{m+l}v_{S_{k-a}^{p}} = \bigoplus_{U\in\, \mathcal{U}}Fv_U~.
\]
Since
$
    T_a = S_{k-a}^{p}\cup \{t'_1,\, \ldots, t'_a\},
$
the map of $\mathcal{T}$ to $\mathcal{U}$ defined by
\[
    T \mapsto T\cap\{1,\,\cdots,\,m+l\} = \{t_1,\,\ldots,\,t_{k-a}\}
\]
is one-to-one and onto. Since $B_{m+l}$ fixes $\{t'_1,\, \ldots, t'_a\}$ pointwise, the map defined by
$
    v_T \mapsto v_{\{t_1,\, \ldots,\, t_{k-a}\}}
$
leads to a $B_{m+l}$-module isomorphism of $B_{m+l}v_{T_a}$ onto $B_{m+l}v_{S_{k-a}^{p}}$, which is indecomposable by Proposition \ref{indDecom}.

Since there are ${k-l \choose a}$ ways to choose $\{t'_1,\,\ldots,\, t'_a\}$
from $\{m+l+1,\ldots,m+k\}$, there are the same number of corresponding
modules $B_{m+l}v_{T_a}$ in $W^m_k$. Thus
\[
W^m_k \downarrow_{B_{m+l}}^{B_{m+k}} \,\cong\, \bigoplus_{a=0}^{k-l} {k-l \choose a}B_{m+l}v_{T_a}~,
\]
and the proof is complete.
\end{proof}

\begin{corollary} Let $k, l, m$ be positive integers with $l<k$.
    \begin{equation}\label{moddecomform}
        {m+k \choose k}=\sum_{a=0}^{k-l} {k-l \choose a}{m+l \choose k-a}~.
    \end{equation}
\end{corollary}
\begin{proof}
  Corollary \ref{dimspeci2} shows that $\dim W^m_k={m+k \choose k}$ and $\dim W_{k-a}^{m+l-k+a} = {m+l \choose k-a}$. Inserting them into (\ref{moddecom1}), we complete the proof.
\end{proof}

We apply the same procedure in the proof of Proposition \ref{decomposition} to deal with, without showing all the details, the decomposition of the $B_{2k}$-module $W_k=B_nv_{S_k}$ into indecomposable $B_{2(k-1)}$-modules for $S_k=\{2,4,\ldots,2k\}$ with $2k\le n$.  We regard $W_k$ as a $B_{2(k-1)}$-module.

Let $S$ be a $k$-subset of $\mathbf{n}$ such that $S\leq S_k$. If $S$ contains $2k$, then $v_S=fv_{S_k}$ for some $f\in B_{2(k-1)}$ since $B_{2(k-1)}$ fixes $2k$. Thus the module $B_{2(k-1)}v_{S_k}$ contains  basis vectors $v_S$ where $S$ runs through all $k$-subsets containing $2k$, so $B_{2(k-1)}v_{S_k}$ is isomorphic to the submodule
$W_{k-1} = B_{2(k-1)}v_{S_{k-1}}$.

If $S$ does not contain $2k$ but contains $2k-1$, then $S\leq T=\{2, \,4,\,\ldots,\,2(k-1),\,2k-1\}$, and since $B_{2(k-1)}$
fixes $2k-1$, we have $v_S=fv_{T}$ for some $f\in B_{2(k-1)}$. Thus the module $B_{2(k-1)}v_{T}$ contains basis vectors $v_S$ where $S$ runs through all $k$-subsets containing $2k-1$, so it is isomorphic to the submodule $W_{k-1} = B_{2(k-1)}v_{S_{k-1}}$.

If $S$ contains neither $2k$ nor $2k-1$, then $S\leq S_{k,\, k -2}=\{2,4,\ldots,2(k-2),2k-3,2k-2\}$, and we have
$v_S=fv_{S_{k,\, k -2}}$ for some $f\in B_{2(k-1)}$. The module $B_{2(k-1)}v_{k,\, k -2}$ then contains basis vectors $v_S$ where $S$ runs through all $k$-subsets containing neither $2k$ nor $2k-1$. Therefore, we have shown
\begin{proposition}
For the $k$-subset $S_k=\{2,4,\ldots,2k\}$ of $\mathbf{n}$, let
$W_k=B_nv_{S_k}$. Then viewed as a $B_{2(k-1)}$ module, we have the
decomposition of $W_k$ into a direct sum of indecomposable
submodules
\begin{equation}\label{moddecom2}
W_k\downarrow_{B_{2(k-1)}}^{B_{2k}}\cong 2W_{k-1} \oplus
B_{2(k-1)}v_{S_{k,\,k-2}}~.
\end{equation}
\end{proposition}

\begin{corollary}\label{ccd} Let $k\ge 2$. Then the $(k+1)${\rm st} Catalan number
    \begin{equation}\label{ccd} c_{k+1}=2c_k + d_{k,\,k-2}~,
    \end{equation}
    where $c_k$ is the $k$th Catalan number and $d_{k,\,k-2}$ is the dimension of $B_{2(k-1)}v_{S_{k,\,k-2}}$.
\end{corollary}
\begin{proof}
By Corollary \ref{dimspeci1}, the dimension of $W_k$ is $c_{k+1}$ and that of $W_{k-1}$ is $c_{k}$.
The dimension $d_{k,\,k-2}$ of $B_{2(k-1)}v_{S_{k,\,k-2}}$ is given in (\ref{catagene2}). Putting them into (\ref{moddecom2}), we obtain the desired combinatorial identity (\ref{ccd}).
\end{proof}

\section{Presentation on generators and relations}
We use the method of Section 3 of Herbig \cite{H} to describe generators and relations of $B_n$.
Some preparations are needed. We define a new monoid $\hat{B}_n$ generated by symbols $\l_i,\e_i, \hat 1$ subject to the relations:
\vspace{-2mm}
\begin{enumerate}[{\rm(i)}]
  \item $\e_i^2=\e_i$
  \vspace{-2mm}
  \item $\l_i\l_{i+1}\l_i=\l_i\l_{i+1}=\l_{i+1}\l_i\l_{i+1}$
  \vspace{-2mm}
  \item $\l_i\e_i=\l_i=\e_{i+1}\l_i$
  \vspace{-2mm}
  \item $\l_i\e_{i+1}=\e_i\e_{i+1}=\e_i\l_i=\l_i^3=\l_i^2$
  \vspace{-2mm}
  \item $\e_i\l_j=\l_j\e_i$ for $i\neq j,j+1$
  \vspace{-2mm}
  \item $\l_i\l_j=\l_j\l_i$ for $|i-j|\geq 2$
  \vspace{-2mm}
  \item $\e_i\e_j=\e_j\e_i$ for all $i,j$~.
\end{enumerate}
\vspace{-2mm}

For $a, b\in \mathbf{n}$ with $a>b$, let $\L^{a,\,a}=\hat{1}$ and
$
    \L^{a,\,b}=\l_b\l_{b+1}\cdots \l_{a-2}\l_{a-1}.
$
For any subsets $S=\{s_1<\dots<s_k\}$ and $T=\{t_1<\dots<t_k\}$ of $\mathbf{n}$ satisfying $s_j\geq t_j$ for all $1\leq j\leq k$,
let
\begin{equation}\label{uv}
  U=\mathbf{n}-S=\{u_1<\cdots<u_{n-k}\}\quad\text{and}\quad V=\mathbf{n}-T=\{v_1<\cdots<v_{n-k}\}~.
\end{equation}
Define
$
   \E_S = \e_{u_1}\cdots \e_{u_{n-k}},\,\L^{S,\,T} =\L^{s_k,\,t_k}\cdots \L^{s_1,\,t_1},\text{ and } \E_T=\e_{v_1}\cdots \e_{v_{n-k}}~,
$
where we agree that $\E_\mathbf{n} = \hat{1}$. The word $W^S_T=\E_T\L^{S,\,T}\E_S$ is called a {\em standard word} of $\hat{B}_n$. The following proposition shows that every element of $\hat{B}_n$ is equivalent, under the relations (i) -- (vii), to one of the standard words. First, we note that the generators $\l_i,\e_i$ and $\hat{1}$ are themselves standard words.
\begin{proposition}\label{ml}
\begin{enumerate}[{\rm(1)}]
  \item $W^S_T\l_i=W^{S^\prime}_{T^\prime},$ where
$$(S^\prime,\,T^\prime)=
        \left\{
         \begin{array}{ll}
         (S,\,T) & \text{\rm if }\, i,\,i+1\notin S \\
         (S\backslash\{i\},\,T\backslash\{t_{c+1}\}) & \text{\rm if }\, i,\,i+1\in S  \\
          ((S\backslash\{i\})\cup\{i+1\},\,T) & \text{\rm if }\, i\in S,\,i+1\notin S \\
          (S\backslash\{i+1\},\,T\backslash\{t_{c+1}\}) & \text{\rm if }\, i\notin S,\,i+1\in S \\
          \end{array}
          \right.
        $$where $i+1$ is mapped to $t_{c+1}$ under $W^S_T$.

  \item $W^S_T\e_i=W^{S^\prime}_{T^\prime},$ where
$$(S^\prime,\,T^\prime)=
        \left\{
         \begin{array}{ll}
         (S,\,T) & \text{\rm if }\, i\notin S \\
          (S\backslash\{i\},\,T\backslash\{t_c\})& \text{\rm if }\, i\in S \\
          \end{array}
          \right.
        $$where $i$ is mapped to $t_c$ under $W^S_T$.

  \item $\l_iW^S_T=W^{S^\prime}_{T^\prime},$ where
$$(S^\prime,\,T^\prime)=
        \left\{
         \begin{array}{ll}
         (S,\,T) & \text{\rm if }\, i,\,i+1\notin T \\
         (S\backslash\{s_c\},\,T\backslash\{i+1\}) & \text{\rm if }\, i,\,i+1\in T \\
         (S\backslash\{s_c\},\,T\backslash\{i\}) & \text{\rm if }\, i\in T,\,i+1\notin T \\
          (S,\,(T\backslash\{i+1\})\cup\{i\}) & \text{\rm if }\, i\notin T,\,i+1\in T \\
          \end{array}
          \right.
        $$where $s_c$ is mapped to $i$ under $W^S_T$.
  \item $\e_iW^S_T=W^{S^\prime}_{T^\prime},$ where
$$(S^\prime,\,T^\prime)=
        \left\{
         \begin{array}{ll}
         (S,\,T) & \text{\rm if }\, i\notin T  \\
         (S\backslash\{s_c\},\,T\backslash\{i\}) & \text{\rm if }\, i\in T \\
         \end{array}
          \right.
        $$where $s_c$ is mapped to $i$ under $W^S_T$.
\end{enumerate}
\end{proposition}
\begin{proof}

We use relations (i) to (vii) repeatedly, but sometimes we do not mention them.
We divide the proof of part (1) into four cases.

Case 1.1: neither $i$ nor $i+1$ is in $S$. Then $i,i+1\in U$ and $\hat{e_i},\e_{i+1}$ appear in $\E_S$. Since the items in $\E_S$ commute with $\hat{l_i}$ except $\hat{e_i},\e_{i+1}$, we have
\begin{eqnarray*}
W^S_T\l_i &=& \E_T\L^{S,\,T}\e_{u_1}c\dots \e_i\e_{i+1}\cdots
              \e_{u_{n-k}}\l_i \\
   &=& \E_T\L^{S,\,T}\e_{u_1}\cdots\e_i\e_{i+1}\l_i\cdots
       \e_{u_{n-k}}  \\
   &=& \E_T\L^{S,\,T}\e_{u_1}\cdots\e_i\l_i\cdots \e_{u_{n-k}}  \quad\quad\text{(by \rm (iii))}\\
   &=& \E_T\L^{S,\,T}\e_{u_1}\cdots\e_i\e_{i+1}\cdots \e_{u_{n-k}} \quad\text{(by \rm (iv))} \\
   &=& W^S_T~.
\end{eqnarray*}
Case 1.2: both $i$ and $i+1$ are in $S$. We have $i,i+1\notin U$, and every term in $\E_S$ commutes with $\l_i$ by (v).
From (iii) we get
$$
    \E_S\l_i=\E_S\l_i\e_i =\l_i\E_S\e_i=\l_i\E_{S^\prime}~,
$$
where $S^\prime=S\backslash\{i\}.$
Let $s_c=i$. Then $s_{c+1}=i+1$ since $i+1\in S$. For $j<c$, by (vi) the terms in $\L^{s_j,t_j}$ commute with $\l_i$, as
their indices are less than or equal to $i-2$. Thus
\begin{eqnarray}\label{wstl}
\nonumber  W^S_T\l_i &=& \E_T\L^{S,\,T}\E_S\l_i \\
\nonumber   &=& \E_T\L^{S,\,T}\l_i\E_{S^\p} \\
   &=& \E_T\L^{s_k,\,t_k}\cdots\L^{s_{c+2},\,t_{c+2}}(\L^{i+1,\,t_{c+1}}\L^{i,\,t_{c}}\l_i)\L^{s_{c-1},\,t_{c-1}}\cdots\L^{s_1,\,t_1}\E_{S^\p}~.
\end{eqnarray}
We now show
\begin{equation}\label{lll}
    \L^{i+1,\,t_{c+1}}\L^{i,\,t_c}\l_i = L^{i+1,\, t_{c}}~.
\end{equation}
Indeed, by (vi) the term $\l_i$ in $\L^{i+1,\,t_{c+1}}=\l_{t_{c+1}}\cdots\l_{i-1}\l_i$ commutes with all of the terms of $\L^{i,\,t_{c}}=\l_{t_c}\cdots\l_{i-2}\l_{i-1}$ until $\l_{i-1}$, where we use (ii): $\l_i\l_{i-1}\l_i=\l_{i-1}\l_i$ to simplify.
Repeating the same procedure for each of the remaining terms of $\L^{i+1,\,t_{c+1}}$, we conclude
\begin{eqnarray*}
   \L^{i+1,\,t_{c+1}}\L^{i,\,t_c}\l_i &=& (\l_{t_{c+1}}\cdots\l_{i-1}\l_i)(\l_{t_c}\cdots\l_{i-1})\l_i \\
                                                &=& (\l_{t_{c+1}}\cdots\l_{i-1})(\l_{t_c}\cdots\l_i\l_{i-1}\l_i) \quad\text{\rm by (vi)} \\
                   &=& (\l_{t_{c+1}}\cdots\l_{i-1})(\l_{t_c}\cdots\l_{i-1}\l_i)  \quad\text{\rm by (ii)}\\
                    &\vdots&  \\
                    &=& \l_{t_c}\cdots\l_{i-1}\l_i \\
                    &=& L^{i+1,\, t_{c}}~.
  \end{eqnarray*}
Putting (\ref{lll}) into (\ref{wstl}), we obtain
\begin{eqnarray}\label{wstl1}
\nonumber W^S_T\l_i &=&    \E_T\L^{s_k,\,t_k}\cdots\L^{s_{c+2},\,t_{c+2}}\L^{i+1,\,t_{c}}\L^{s_{c-1},\,t_{c-1}}\cdots\L^{s_1,\,t_1}\E_{S^\p}\\
&=& \E_T\L^{S^\prime,\,T^\p}\E_{S^\prime}
\end{eqnarray}
where $S'=S\setminus\{i\}$ and $T'=T\setminus\{t_{c+1}\}$.

The right hand side of (\ref{wstl1}) is not a standard word yet because $\e_{t_{c+1}}$ is still missing in $\E_T$. Since
$t_c\leq t_{c+1}-1\leq i$, the term $\l_{(t_{c+1}-1)}$ appears in $\L^{i+1,\,t_{c}}=\l_{t_c}\cdots\l_i$. Notice that
$\l_{(t_{c+1}-2)}\l_{(t_{c+1}-1)}=\l_{(t_{c+1}-1)}\l_{(t_{c+1}-2)}\l_{(t_{c+1}-1)}$ and $\e_{t_{c+1}}\l_{(t_{c+1}-1)}=\l_{(t_{c+1}-1)}$. We have
\begin{eqnarray*}
L^{i+1,\, t_{c}}  &=& \l_{t_c}\cdots\l_{(t_{c+1}-3)}(\l_{(t_{c+1}-2)}\l_{(t_{c+1}-1)}) \cdots\l_i \\
   &=& \l_{t_c}\cdots\l_{(t_{c+1}-3)}(\l_{(t_{c+1}-1)}\l_{(t_{c+1}-2)} \l_{(t_{c+1}-1)}) \cdots\l_i \\
   &=& \l_{t_c}\cdots\l_{(t_{c+1}-3)}(\e_{t_{c+1}}\l_{(t_{c+1}-1)})\l_{(t_{c+1}-2)}\l_{(t_{c+1}-1)}  \cdots\l_i \\
   &=& \l_{t_c}\cdots\l_{(t_{c+1}-3)}\e_{t_{c+1}} \l_{(t_{c+1}-2)}\l_{(t_{c+1}-1)} \cdots\l_i\\
   &=& \e_{t_{c+1}}(\l_{t_c}\cdots \l_{(t_{c+1}-3)}\l_{(t_{c+1}-2)}\l_{(t_{c+1}-1)} \cdots\l_i)\\
   &=& \e_{t_{c+1}}L^{i+1,\, t_{c}}~,
\end{eqnarray*}
where we have used that $\e_{t_{c+1}}$ commutes with all terms on its left.
Inserting the above result into (\ref{wstl1}) and knowing that $\e_{t_{c+1}}$ commutes with $\L^{s_k,\,t_k}\cdots\L^{s_{c+2},t_{c+2}}$, we deduce
\begin{eqnarray*}
W^S_T\l_i &=& \E_T (\L^{s_k,\,t_k}\cdots\L^{s_{c+2},\,t_{c+2}})(\e_{t_{c+1}}L^{i+1,\, t_{c}}) \L^{s_{c-1},\,t_{c-1}}\cdots\L^{s_1,\,t_1}\E_{S^\prime} \\
&=& \E_T\e_{t_{c+1}}(\L^{s_k,\,t_k}\cdots\L^{s_{c+2},\,t_{c+2}} L^{i+1,\, t_{c}} \L^{s_{c-1},\,t_{c-1}}\cdots\L^{s_1,\,t_1})\E_{S^\prime} \\   &=&\E_{T^\p}\L^{S^\p,\,T^\p}\E_{S^\p}  \\
   &=&W^{S^\p}_{T^\p}~.
\end{eqnarray*}
Case 1.3: $i$ is in $S$, but $i+1$ is not. It follows immediately that $i\notin U,\,i+1\in
U$, so $\e_{i+1}$ appears in $\E_{S}$, but $\e_i$ does not. Using (iii): $\e_{i+1}\l_i=\l_i\e_i$ and (v): $\l_i\e_j=\e_j\l_i$ for $j\ne i, i+1$, we have
\begin{eqnarray*}
\E_S\l_i&=&\e_{u_1}\cdots\e_{i+1}\cdots\e_{u_{n-k}}\l_i\\
   &=& \e_{u_1}\cdots\e_{i+1}\l_i\cdots\e_{u_{n-k}} \\
   &=& \e_{u_1}\cdots\l_i\e_i\cdots\e_{u_{n-k}} \\
   &=& \l_i\E_{S^\p}~,
\end{eqnarray*}
where $S^\p=(S\backslash\{i\})\cup\{i+1\}$. Let $s_c=i$. Then $s_{c+1}>i+1$ since $i+1\notin S$. For
$j<c$, all the terms in $\L^{s_j,t_j}$ commute with $\l_i$ since their indices are at most $i-2$, giving
\begin{eqnarray*}
W^S_T\l_i &=& \E_T\L^{S,\,T}\E_S\l_i  \\
   &=&\E_T\L^{s_k,\,t_k}\cdots\L^{i,t_c}
      \cdots\L^{s_1,\,t_1}\l_i\E_{S^\p}\\
   &=&\E_T\L^{s_k,\,t_k}\cdots(\l_{t_c}\cdots\l_{i-2}\l_{i-1})\l_i
      \cdots\L^{s_1,\,t_1}\E_{S^\p}\\
   &=&\E_{T^\p}\L^{S^\p,\,T^\p}\E_{S^\p}\\
   &=&W^{S^\p}_{T^\p}~,
\end{eqnarray*}
where $T'=T$.

Case 1.4: $i+1$ is in $S$, but $i$ is not. We have $i\in U,i+1\notin U$. The item $\e_i$ is in $\E_S$, but $\e_{i+1}$ is not. By $\e_i\l_i=\l_i\e_{i+1}$, $\l_i=\l_i\e_i$, and  $\l_i\e_j=\e_j\l_i$ for $j\ne i, i+1$, we find
\begin{eqnarray*}
\E_S\l_i&=&\e_{u_1}\cdots\e_i\cdots\e_{u_{n-k}}\l_i\\
   &=& \e_{u_1}\cdots\e_i\l_i\cdots\e_{u_{n-k}} \\
   &=& \e_{u_1}\cdots\l_i\e_{i+1}\cdots\e_{u_{n-k}} \\
   &=& \e_{u_1}\cdots\l_i\e_i\e_{i+1}\cdots\e_{u_{n-k}} \\
   &=& \l_i\e_{u_1}\cdots\e_i\e_{i+1}\cdots\e_{u_{n-k}} \\
   &=& \l_i\E_{S^\p}~,
\end{eqnarray*}
where $S^\p=S\backslash\{i+1\}$.

Let $s_{c+1}=i+1$. Then for $j\le c$, the indices of the terms in $\L^{s_j,\,t_j}$ are at most $i-2$, so they commute with
$\l_i$, leading to
\begin{eqnarray*}
W^S_T\l_i &=& \E_T\L^{S,\,T}\E_S\l_i  \\
   &=&\E_T\L^{s_k,\,t_k}\cdots\L^{s_{c+2},\,t_{c+2}}\L^{i+1,\,t_{c+1}}\L^{s_{c},\,t_{c}}\cdots\L^{s_1,\,t_1}\l_i\E_{S^\p}  \\
   &=&\E_T\L^{s_k,\,t_k}\cdots\L^{s_{c+2},\,t_{c+2}}(\l_{t_{c+1}}\cdots \l_{i-1}\l_i)\l_i\L^{s_{c},\,t_{c}} \cdots\L^{s_1,\,t_1}\E_{S^\p}~.
\end{eqnarray*}
Since $\l_i^2=\e_i\e_{i+1}$ and $\l_{j-1}\e_j=\e_{j-1}\e_j$ for all
$t_{c+1}+1\leq j\leq i$ (use them repeatedly below), we obtain
\begin{eqnarray*}
W^S_T\l_i&=&\E_T\L^{s_k,\,t_k}\cdots\L^{s_{c+2},\,t_{c+2}}(\l_{t_{c+1}}\cdots
            \l_{i-1}\l_i^2)\L^{s_{c},\,t_{c}} \cdots\L^{s_1,\,t_1}\E_{S^\p}\\
   &=&\E_T\L^{s_k,\,t_k}\cdots\L^{s_{c+2},\,t_{c+2}}(\l_{t_{c+1}}\cdots
            \l_{i-1}\e_i\e_{i+1})\L^{s_{c},\,t_{c}}\cdots\L^{s_1,\,t_1}\E_{S^\p}\\
   &=&\E_T\L^{s_k,\,t_k}\cdots\L^{s_{c+2},\,t_{c+2}}(\l_{t_{c+1}}\cdots
            \l_{i-2}\e_{i-1}\e_i\e_{i+1})\L^{s_{c},\,t_{c}}\cdots\L^{s_1,\,t_1}\E_{S^\p}\\
   &\vdots&  \\
   &=&\E_T\L^{s_k,\,t_k}\cdots\L^{s_{c+2},\,t_{c+2}}(\e_{t_{c+1}}\cdots
            \e_{i-1}\e_i\e_{i+1})\L^{s_{c},\,t_{c}}\cdots\L^{s_1,\,t_1}\E_{S^\p}~.
\end{eqnarray*}
As the indices of all the terms $\l$ on the left of $\e_{t_{c+1}}$ are at least $t_{c+1}+1$, it follows from (v) that $\e_{t_{c+1}}$ commutes with $\L^{s_k,\,t_k}\cdots\L^{s_{c+2},\,t_{c+2}}.$ We obtain
\begin{eqnarray*}
W^S_T\l_i  &=&\E_T\e_{t_{c+1}}\L^{s_k,\,t_k}\cdots\L^{s_{c+2},\,t_{c+2}}
            (\e_{t_{c+1}+1}\cdots\e_i\e_{i+1})
            \L^{s_{c},\,t_{c}}\cdots\L^{s_1,\,t_1}\E_{S^\p}\\
   &=&\E_{T^\p}\L^{s_k,\,t_k}\cdots\L^{s_{c+2},\,t_{c+2}}
            (\e_{t_{c+1}+1}\cdots\e_i\e_{i+1})
            \L^{s_{c},\,t_{c}}\cdots\L^{s_1,\,t_1}\E_{S^\p}~,
\end{eqnarray*}
where $T^\p=T\backslash\{t_{c+1}\}$.

Our aim now is to switch $\e_j$ for $t_{c+1}+1\leq j\leq i+1$ with the terms $\l$ one by one on the right of $\e_j$ until it encounters either $\l_j$ or $\l_{j-1}$, or until it commutes past all of the $\l$. If $t_{c+1}+1\leq j\leq s_c$, then $\e_j$ will run into some $\l_{j-1}$ in $\L^{s_c,\,t_c}=\l_{t_c}\l_{t_c+1}\cdots\l_{s_c-1}$, leading to $\e_j\l_{j-1} = \l_{j-1}$ by (iii).
If $s_c+1\le j\leq i+1$, then $\e_j$ does not run into any $\l_j$ or $\l_{j-1}$ in $\L^{s_c,\,t_c}$ since the maximal index of the $\l$ there is $s_c-1$, and $\e_j$ does not run into any $\l_j$ or $\l_{j-1}$ to the right of $\L^{s_c,\,t_c}$, either, because all the indices of the $\l$ are at most $s_c-2$. But for $s_{c+1}\le j\leq i+1$, we have $j\in \mathbf{n}-S^\p$, so $\e_j$ will run into $\e_j$ in $\E_{S^\p}$. By $\e_j\l_{j-1}=\l_{j-1}$ and
$\e^2_j=\e_j$, we get
\begin{eqnarray*}
W^S_T\l_i&=&\E_{T^\p}\L^{s_k,\,t_k}\cdots\L^{s_{c+2},\,t_{c+2}}
              \L^{s_c,\,t_c}
            \cdots\L^{s_1,\,t_1}\E_{S^\p}  \\
   &=&\E_{T^\p}\L^{S^\p,\,T^\p}\E_{S^\p}  \\
   &=&W^{S^\p}_{T^\p}~.
\end{eqnarray*}
This completes the proof of part (1).

We next show part (2). If $i\notin S$, then $i\in U$ and $\e_i$ is in $\E_S$. By $\e^2_i=\e_i$, we have
\begin{eqnarray*}
W^S_T\e_i &=&\E_T\L^{S,\,T}\e_{u_1}
             \cdots\e_i\cdots\e_{u_{n-k}}\e_i \\
   &=&\E_T\L^{S,\,T}\e_{u_1}\cdots\e_i\e_i\cdots\e_{u_{n-k}} \\
   &=&\E_T\L^{S,\,T}\e_{u_1}\cdots\e_i\cdots\e_{u_{n-k}}\\
   &=& \E_T\L^{S,\,T}\E_S \\
   &=& W^S_T~.
\end{eqnarray*}
If $i\in S$, let $s_c=i$, and then for $j<c$, all of the
indices of the terms in $\L^{s_j,\,t_j}$ are at most $i-2$, so they
commute with $\e_i$. From $\e_i^2=\e_i$, we obtain
\begin{eqnarray*}
W^S_T\e_i&=&\E_T\L^{s_k,\,t_k}\cdots\L^{i,\,t_c}
            \cdots\L^{s_1,\,t_1}\E_S\e_i  \\
   &=&\E_T\L^{s_k,\,t_k}\cdots\L^{i,\,t_c}
            \cdots\L^{s_1,\,t_1}\E_S\e_i\e_i  \\
   &=&\E_T\L^{s_k,\,t_k}\cdots\L^{i,\,t_c}\e_i
            \cdots\L^{s_1,\,t_1}\E_{S^\p}~.
\end{eqnarray*}
where $S^\p=S\backslash\{i\}$.
 For $t_c+1\leq j\leq i$, using (iv):
$\l_{j-1}\e_j=\e_{j-1}\e_j$ repeatedly, we get
\begin{eqnarray*}
W^S_T\e_i&=&\E_T\L^{s_k,\,t_k}\cdots(\l_{t_c}\l_{t_c+1}\cdots
            \l_{i-1})\e_i\cdots\L^{s_1,\,t_1}\E_{S^\p}\\
   &=&\E_T\L^{s_k,\,t_k}\cdots(\l_{t_c}\l_{t_c+1}\cdots\l_{i-2}
            \e_{i-1}\e_i)\cdots\L^{s_1,\,t_1}\E_{S^\p}\\
   &\vdots&  \\
   &=&\E_T\L^{s_k,\,t_k}\cdots(\e_{t_c}\e_{t_c+1}\cdots
            \e_{i-1}\e_i)\cdots\L^{s_1,\,t_1}\E_{S^\p}~.
\end{eqnarray*}
Now the situation is very similar to Case 1.3. We commute $\e_{t_c}$ past all $\l$ on the left and get
$\E_T\e_{t_c}=\E_{T^\p}$, where $T^\p=T\backslash\{t_c\}$.
For
$t_c+1\leq j\leq s_c-1$, commute $\e_j$ past all $\l$ on the right
until they reach an $\l_{j-1}$, then use $\e_j\l_{j-1}=\l_{j-1}$. For $s_{c}-2\le j\leq i$, commute $\e_j$ past all $\l$ on the right and until it meets an $\e_j$ in $\E_{S^\p}$, then use $\e_j^2=\e_j$. We find
\begin{eqnarray*}
W^S_T\e_i&=&\E_T\L^{s_k,\,t_k}\cdots(\e_{t_c}\cdots
            \e_{i-1}\e_i)\cdots\L^{s_1,\,t_1}\E_{S^\p}\\
         &=&\E_T\e_{t_c}\L^{s_k,\,t_k}\cdots(\e_{t_c+1}\cdots
            \e_{i-1}\e_i)\cdots\L^{s_1,\,t_1}\E_{S^\p}\\
         &=&\E_{T^\p}L^{S^\p,\,T^\p}E_{S^\p}\\
         &=&W^{S^\p}_{T^\p}~.
\end{eqnarray*}

Parts (3) and (4) are similar.
\end{proof}

\begin{remark}
{\rm
One can prove that the standard words can also be chosen as $\E_{S,\,T}\L^{S,\,T}$ where $\E_{S,\,T}=\e_{w_1}\e_{w_2}\dots\e_{w_h}$ if $\mathbf{n}-S-T=\{w_1,w_2,\dots,w_h\}$, and $L^{S,\,T}$ is the same as Proposition \ref{ml}.
}
\end{remark}

Let $G$ be the subset of $B_n$ consisting of the elements $1$, $l_i$, and $e_j$, where
\begin{align*}
  l_i&=\left(
        \begin{array}{ccccccc}
          1 &  \dots & i-1 & i+1 & i+2 & \dots  & n \\
          1 &  \dots & i-1 & i & i+2 & \dots  & n \\
        \end{array}
      \right)\quad \text{ and} \\
  e_j&=\left(
        \begin{array}{cccccc}
          1  & \dots & j-1 &  j+1 & \dots  & n \\
          1  & \dots & j-1 &  j+1 & \dots  & n \\
        \end{array}
      \right)~,
\end{align*}
where $1\le i\leq n-1$ and $1\leq j\leq n$.

Our intention here is to show that $G$ generates $B_n$. For any arbitrary $a, b\in \mathbf{n}$ with $a>b$, we define $L^{a,\,b}=l_b\cdots l_{a-2}l_{a-1}$
and $L^{a,\,a}=1$.
For $f\in B_n$, let
\[
    S=\{s_1<\cdots<s_k\} \quad\text{and}\quad T=\{t_1<\dots<t_k\}
\]
be its respective domain and range. Let $U, V$ be as in (\ref{uv}) and let
\[
    E_S=e_{u_1}\cdots e_{u_{n-k}},\quad L^{S,\,T}=L^{s_k,\,t_k}\cdots L^{s_1,\,t_1},\quad E_T=e_{v_1}\cdots e_{v_{n-k}}~,
\]
where we agree that $E_\mathbf{n}=1$. Then $x=E_TL^{S,\,T}E_S$, and we have shown
\begin{theorem}\label{gene}
Every element of $B_n$ is a product of elements of $G$.
\end{theorem}

The elements of $G$ satisfy the following relations (we omit the details, which are straightforward), where the indices $i$ and $j$ are such that all expressions in the relations are meaningful.

We use $R$ to denote the set of these relations:
\vspace{-2mm}
\begin{enumerate}[{\rm(1)}]
  \item $e_i^2=e_i$~.
\vspace{-2mm}
  \item $l_il_{i+1}l_i=l_il_{i+1}=l_{i+1}l_il_{i+1}$~.
\vspace{-2mm}
  \item $l_ie_i=l_i=e_{i+1}l_i$~.
\vspace{-2mm}
  \item $l_ie_{i+1}=e_ie_{i+1}=e_il_i=l_i^3=l_i^2$~.
\vspace{-2mm}
  \item $e_il_j=l_je_i$ \quad for $i\neq j,j+1$~.
\vspace{-2mm}
  \item $l_il_j=l_jl_i$ \quad for $|i-j|\geq 2$~.
\vspace{-2mm}
  \item $e_ie_j=e_je_i$ \quad for all $i,j$~.
\end{enumerate}
\begin{theorem}\label{main}
The monoid $B_n$ has presentation $\langle\, G \mid R\,\rangle $.
\end{theorem}

\begin{proof}
The mapping $\phi:\hat{B}_n\rightarrow B_n$ defined by $\l_i\mapsto l_i$ and $\e_i\mapsto e_i$ and $\hat 1\mapsto 1$ induces a monoid homomorphism of $\hat{B}_n$ onto $B_n$. It follows from Proposition \ref{ml} that the mapping is injective.
\end{proof}

\end{document}